\documentclass[12pt]{amsart}
\usepackage{amsfonts,amssymb,amscd,amsmath,enumerate,verbatim}
\usepackage[latin1]{inputenc}
\usepackage{amscd}
\usepackage{latexsym}
\usepackage{pstcol,pst-plot,pst-3d}
\usepackage{mathptmx}
\usepackage{multicol}

\psset{unit=0.7cm,linewidth=0.8pt,arrowsize=2.5pt 4}

\newpsstyle{fatline}{linewidth=1.5pt}
\newpsstyle{fyp}{fillstyle=solid,fillcolor=verylight}
\definecolor{verylight}{gray}{0.97}
\definecolor{light}{gray}{0.9}
\definecolor{medium}{gray}{0.85}

\def\frk{\mathfrak}               

\def\Phi{{\frk N}}
\def\opn#1#2{\def#1{\operatorname{#2}}} 
\opn\chara{char} \opn\length{\ell} \opn\pd{pd} \opn\rk{rk}
\opn\projdim{proj\,dim} \opn\injdim{inj\,dim} \opn\rank{rank}
\opn\depth{depth} \opn\grade{grade} \opn\height{height}
\opn\embdim{emb\,dim} \opn\codim{codim}

\opn\Tr{Tr} \opn\bigrank{big\,rank}
\opn\superheight{superheight}\opn\lcm{lcm}
\opn\trdeg{tr\,deg}
\opn\reg{reg} \opn\lreg{lreg} \opn\ini{in} \opn\lpd{lpd}
\opn\size{size}\opn{\mult}{mult}
\opn\div{div} \opn\Div{Div} \opn\cl{cl} \opn\Cl{Cl}
\opn\Spec{Spec} \opn\Supp{Supp} \opn\supp{supp} \opn\Sing{Sing}
\opn\Ass{Ass} \opn\Min{Min}
\opn\Ann{Ann} \opn\Rad{Rad} \opn\Soc{Soc}
\opn\Syz{Syz} \opn\Im{Im} \opn\Ker{Ker} \opn\Coker{Coker}
\opn\Am{Am} \opn\Hom{Hom} \opn\Tor{Tor} \opn\Ext{Ext}
\opn\End{End} \opn\Aut{Aut} \opn\id{id} \opn\ini{in}

\opn\nat{nat}
\opn\pff{pf}
\opn\Pf{Pf} \opn\GL{GL} \opn\SL{SL} \opn\mod{mod} \opn\ord{ord}
\opn\Gin{Gin}
\opn\Hilb{Hilb}\opn\adeg{adeg}\opn\std{std}\opn\ip{infpt}
\opn\Pol{Pol}
\opn\sat{sat}
\opn\Var{Var}
\opn\Gen{Gen}
\opn\indmatch{indmatch}

\opn\aff{aff} \opn\con{conv} \opn\relint{relint} \opn\st{st}
\opn\lk{lk} \opn\cn{cn} \opn\core{core} \opn\vol{vol}
\opn\link{link} \opn\star{star}
\opn\gr{gr}

\def\pot#1#2{#1[\kern-0.28ex[#2]\kern-0.28ex]}

\opn\dirlim{\underrightarrow{\lim}}
\opn\inivlim{\underleftarrow{\lim}}
\def\Implies{\ifmmode\Longrightarrow \else
        \unskip${}\Longrightarrow{}$\ignorespaces\fi}
\def\implies{\ifmmode\Rightarrow \else
        \unskip${}\Rightarrow{}$\ignorespaces\fi}
\def\iff{\ifmmode\Longleftrightarrow \else
        \unskip${}\Longleftrightarrow{}$\ignorespaces\fi}

\let\:=\colon
\newtheorem{Theorem}{Theorem}[section]
\newtheorem{Lemma}[Theorem]{Lemma}
\newtheorem{Corollary}[Theorem]{Corollary}

\newtheorem{Remark}[Theorem]{Remark}

\newtheorem{Definition}[Theorem]{Definition}

\let\epsilon\varepsilon
\let\phi=\varphi
\let\kappa=\varkappa
\def\qed{\ifhmode\textqed\fi
      \ifmmode\ifinner\quad\qedsymbol\else\dispqed\fi\fi}
\def\textqed{\unskip\nobreak\penalty50
       \hskip2em\hbox{}\nobreak\hfil\qedsymbol
       \parfillskip=0pt \finalhyphendemerits=0}
\def\dispqed{\rlap{\qquad\qedsymbol}}

\opn\dis{dis}
\def\pnt{{\raise0.5mm\hbox{\large\bf.}}}

\opn\Lex{Lex}

\newcommand{\inD}[1][\relax]{\def\argone{#1}\def\temprelax{\relax}
  \ifx\argone\temprelax\right.\else\,\middle|#1\right.{}\fi}

\newif\ifbinary
\binarytrue

\begin{document}

\title{On the Generalized Binomial Edge ideals of Generalized Block Graphs}

\author{Faryal Chaudhry, Rida Irfan}

\address{Department of Mathematics and Statistics, The University of Lahore, Lahore, Pakistan} \email{chaudhryfaryal@gmail.com}

\address{Department of Mathematics, COMSATS Institute of Information Technology, Sahiwal, Pakistan}
\email{ridairfan@ciitsahiwal.edu.pk,ridairfan\_88@yahoo.com}

\address{} \email{}

\thanks{This paper was completed while the authors visited the Grigore Moisil Romanian-Turkish joint Laboratory of Mathematical Research hosted by the Faculty of Mathematics and Computer Science, Ovidius University of Constanta, Romania.}

\begin{abstract}
We compute the depth and (give bounds for) the regularity of generalized binomial edge ideals associated with generalized block graphs.
\end{abstract}


\maketitle

\section{Introduction}
Generalized binomial edge ideals were introduced by Rauh in \cite{Rauh}. They are ideals generated by a collection of $2$-minors in a generic matrix. The interest in studying these ideals comes from their connection to conditional independence ideals.

Let $X=(x_{ij})$ be an $m\times n$-matrix of indeterminates and $G$ be a graph on the vertex set $[n]$. The generalized binomial edge ideal $\mathfrak{J}_G$ of $G$ is generated by all the $2$-minors of $X$ of the form $[k,l|i,j]$ where $1\leq k <l\leq m$ and $\{i,j\}$ is an edge of $G$ with $i<j$. When $m=2$, $\mathfrak{J}_G$ coincides with the classical binomial ideal $J_G$ introduced in \cite{HHHKR} and \cite{Oh}.

Generalized binomial edge ideals are a natural extension of the binomial edge ideals considered in \cite{HHHKR} and \cite{Oh}. Some of the properties of the binomial edge ideal $J_G$ extend naturally to its generalization $\mathfrak{J}_G$. For example, as it was proved in \cite{Rauh}, $\mathfrak{J}_G$ is a radical ideal and its minimal primes are determined by the so-called sets with the cut point property of $G$.

From homological point of view, we are interested in studying the resolution of generalized binomial edge ideals and of the numerical data arising from it. There are already many interesting results concerning the invariants of classical binomial edge ideals. For instance, it is known that the regularity of $J_G$ is bounded below by $1+\ell$, where $\ell$ is the length of longest induced path in $G$ and bounded above by the number of vertices of $G$; see \cite{MM}. Other nice results on the homological properties of $J_G$ may be found in \cite{B, CDI,AD, EHH,EZ, KM, KM1, KM2, Sara, Sara2, Sara3, SZ}.

For generalized binomial edge ideals, not so much is known about their resolutions. For example, Madani and Kiani computed in \cite{Sara2} some of the graded Betti numbers of binomial edge ideals associated to a pair of graphs. In particular, they prove that $\mathfrak{J}_G$ has a linear resolution if and only if $m=2$ and $G$ is the complete graph, and $\mathfrak{J}_G$ has linear relations if and only if $G$ is a complete graph.

In this paper, we study the ideal $\mathfrak{J}_G$ where $G$ is a generalized block graph. We show that $\depth(\mathfrak{J}_G)=\depth(\ini_<(\mathfrak{J}_G))$ and we express this depth in terms of the combinatorics of the underlying graph $G$. Here $<$ denotes the lexicographic order on the set of indeterminates ordered naturally, that is, $x_{11}>\cdots> x_{1n}>x_{21}>\cdots> x_{2n}>\cdots>x_{mn}$. Moreover, for $m\geq n$, we show that $\reg{\mathfrak{J}_G}=\reg(\ini_<(\mathfrak{J}_G))=n$, where $n$ is the number of vertices of the graph $G$. When $m<n$, then we provide an upper bound for the regularity of $\ini_<(\mathfrak{J}_G)$ and, therefore, for the regularity of $\mathfrak{J}_G$ as well. Our results generalize the ones obtained in the papers \cite{CDI, EHH, {KM}} for classical binomial edge ideals associated with (generalized) block graphs.

The organization of our paper is as follows. In Section \ref{S2} we recall basic notions of graph theory including the definition of generalized block graphs and review the definitions of depth and regularity. Section \ref{S3} contains our main result, namely Theorem \ref{2.2} and its proof.

In the last part, we derive some consequences of the main theorem. For example, in Corollary \ref{2.3} we particularize Theorem \ref{2.2} to block graphs and in Corollary \ref{2.4} we show that if $G$ is a block graph, then $\mathfrak{J}_G$ is unmixed if and only if $\mathfrak{J}_G$ is Cohen-Maculay if and only if $G$ is a complete graph.

Finally, in Corollary \ref{2.5} we recover Corollary $15$ in \cite{Sara2} which gives the regularity of $\mathfrak{J}_G$ if $G$ is a path graph, but we show more, namely that $\reg\ini_<(\mathfrak{J}_G)$ is equal to $\reg \mathfrak{J}_G$.

\section{Preliminaries}\label{S2}

In this section, we introduce the notation used in this paper and summarize a few results on generalized binomial edge ideals.

Let  $m,n\geq 2$ be integers and let $G$ be an arbitrary simple graph on the vertex set $[n]$. Throughout this paper all the graphs are simple, that is, without loops and multiple edges. We fix a field $K$; let $X=(x_{ij})$ be an $(m\times n)$-matrix of indeterminates, and denote by $S=K[X]$ the polynomial ring in the variables $x_{ij}, ~i=1,\ldots, m$ and $j=1,\ldots, n$.

For $1\leq k<l\leq m$, and $\{i,j\}\in E(G)$, with $1\leq i<j\leq n$, we set

\[p_{ij}^{kl}=[k,l|i,j]=x_{ki}x_{lj}-x_{li}x_{kj}.\]

The ideal $\mathfrak{J}_G=(p_{ij}^{kl}~: ~1\leq k<l\leq m,~\{i,j\}\in E(G))$ is called the {\em generalized binomial edge ideal} of $G$; see \cite{Rauh}.

We first recall some basic definitions from graph theory.  A {\em chordal} graph is a graph without cycles of length greater than or equal to $4.$ A {\em clique} of a graph $G$ is a complete subgraph of $G.$ The cliques of a  graph $G$ form a simplicial complex, $\Delta(G),$ which is called the {\em clique complex } of $G.$ Its facets are the maximal cliques of $G.$ A graph $G$ is a {\em block graph} if and only if it is  chordal and every two maximal cliques have at most one vertex in common. This class was considered in \cite[Theorem 1.1]{EHH}. A chordal graph is called a {\em generalized block graph} if for any three maximal cliques whose intersection is nonempty then intersection of each pair of them is same. In other words, for every $F_i,F_j,F_k\in \Delta(G)$ with the property that $F_i\cap F_j\cap F_k\neq \emptyset$, we have $F_i\cap F_j=F_j\cap F_k=F_i\cap F_k$. This class of graphs was considered in \cite{KM}. Obviously, every block graph is  a generalized block graph.

Let $G$ be a graph. A vertex $i$ of $G$ whose deletion from the graph gives a graph with more connected components than $G$ is called a {\em cut point} of G. A subset $\mathcal{T}\subset [n]$ is said to have the {\em cut point property} for $G$ (cut point set, in brief) if for every $i\in \mathcal{T}$, $c(\mathcal{T}\setminus \{i\})<c(\mathcal{T})$, where $c(\mathcal{T})$ is the number of connected components of the restriction of $G$ to $[n]\setminus \mathcal{T}$. A {\em cut set} of a graph $G$ is a subset of vertices whose deletion increases the number of connected components of $G$. A minimal cut set of $G$ is a cut set which is minimal with respect to inclusion. The {\em clique number} of a graph $G$ is the maximum size of the maximal cliques of $G$. We denote it by $\omega (G)$.

Let $G$ be a generalized block graph. Then $\mathcal{A}_i(G)$ is the collection of cut sets of $G$ of cardinality $i$, where $i=1,\ldots,\omega (G)-1$. We denote $a_i(G)=|\mathcal{A}_i(G)|$. Clearly,  $a_i(G)=0$ for all $i>1$ if and only if $G$ is a block graph.

The clique complex $\Delta(G)$ of a chordal graph $G$ has the property that there exists a {\em leaf order} on its facets.  This means that the facets of $\Delta(G)$ may be ordered as $F_1,\ldots, F_r$ such that, for every $i>1,$ $F_i$ is a leaf of the simplicial complex generated by $F_1,\ldots,F_i$. A {\em leaf} $F$ of a simplicial complex $\Delta$ is a facet of $\Delta$ with the property that there exists another facet of $\Delta$, say $F'$, such that, for every facet $H\neq F$ of $\Delta$, $H\cap F\subseteq F'\cap F.$ Such facet $F'$ is called a branch of $F$.

Let $<$ be the lexicographic order on $S$ induced by the natural order of the variables, that is, $x_{11}>\cdots> x_{1n}>x_{21}\cdots>x_{mn}$. As it was shown in \cite[Theorem 2]{Rauh}, the Gr\"obner basis of $\mathfrak{J}_{G}$ with respect to this order may be given in terms of the admissible paths of $G$. We recall the definition of an admissible path from \cite{HHHKR} (also see \cite{Rauh}).

\begin{Definition}{\em \
Let $i<j$ be two vertices of $G$. A path $i=i_{0},i_{1},\ldots,i_{r-1},i_{r}=j$ from $i$ to $j$ is called  {\em admissible} if the following conditions are fulfilled:
\begin{enumerate}
  \item $i_{k}\neq i_{l}$ for $k\neq l$;
  \item for each $k=1,\ldots,r-1$ on has either $i_{k}<i$ or $i_{k}>j$;
  \item for any proper subset $\{j_{1},\ldots,j_{s}\}$ of $\{i_{1},\ldots,i_{r-1}\}$, the sequence $i,j_{1},\ldots,j_{s},j$ is not a path in $G$.
\end{enumerate}
}
\end{Definition}
According to \cite{Rauh}, a function $\kappa: \{0,\ldots, r\}\rightarrow [m]$ is called {\em $\pi$-antitone} if it satisfies
\[i_s<i_t\Rightarrow \kappa(s)\geq \kappa(t),  ~\text{for all}~ 0\leq s,t\leq r.\]


To any admissible path $\pi: i=i_{0},i_{1},\ldots,i_{r-1},i_{r}=j$, where $i<j$ and any function \\$\kappa: \{0,\ldots, r\}\rightarrow [m]$ one associates the monomial

\[u_{\pi}^\kappa= \prod_{k=1}^{r-1} x_{\kappa(k)i_{k}}.\]

By \cite[Theorem 2]{Rauh}, it follows that the set of binomials

\[\mathcal{G}=\bigcup _{i<j}\{u_{\pi}^\kappa p_{ij}^{\kappa(j)\kappa(i)} : ~i<j, ~\pi~  \text{is an admissible path in G from $i$ to $j$, $\kappa$ is stricly $\pi$-antitone}\}\]
is a reduced Gr$\ddot{o}$bner basis of $\mathfrak{J}_G$ with respect to the lexicographic order. Therefore, the initial ideal of $\mathfrak{J}_{G}$ is

$$(\bigcup _{i<j}\{u_{\pi}^\kappa x_{\kappa(j)i}x_{\kappa(i)j} : ~i<j, ~\pi ~\text{is an admissible path in G from $i$ to $j$, $\kappa$ is strictly $\pi$-antitone}\}).$$

Moreover, since $\ini_<(\mathfrak{J}_{G})$ is a radical ideal, it follows that $\mathfrak{J}_{G}$ is radical as well. Consequently, $\mathfrak{J}_{G}$ is the intersection of its minimal primes.

We now explain how the minimal primes of $\mathfrak{J}_{G}$ can be identified. In \cite[Section 3]{Rauh} (see also \cite{EHHQ}), it is shown that the minimal primes of $\mathfrak{J}_{G}$ are of the form $P_W$ with $W=[m]\times \mathcal{T}$, where $\mathcal{T}\subset [n]$ is a set with the cut point property of $G$.
For a given cut point set $\mathcal{T}\subset [n]$, let $G_1,\ldots,G_{c(\mathcal{T})}$ be the connected components of the restriction of $G$ to $[n]\setminus \mathcal{T}$ and $\widetilde{G}_1,\ldots,\widetilde{G}_{c(\mathcal{T})}$ the complete graphs on the vertex sets $V(G_1),\ldots,V(G_{c(\mathcal{T})})$, respectively.  Let $Q_t$ be the prime ideal
\[Q_t=( p_{ij}^{kl}~:~1\leq k<l\leq m, ~\{i,j\}\in E(\widetilde{G}_t)),~\text{for} ~1\leq t\leq c(\mathcal{T}).\]

If $W=[m]\times \mathcal{T}$, then
\[P_W=(\{x_{ij}~:~(i,j)\in W\},Q_1,\ldots,Q_{c(\mathcal{T})}).\]

We recall now the definition of some homological invariants of a finitely generated $S$-module $M$. Let $M$ be a graded finitely generated $S$-module and
\[\mathbf{F_\bullet} : \cdots\rightarrow F_{2}\rightarrow F_{1}\rightarrow F_{0}\rightarrow M\rightarrow 0\]
be its minimal graded free $S$-resolution with $F_{i}=\bigoplus_{j\in \mathbb{Z}} S(-j)^{\beta_{ij}}$ for all $i$. Then the exponents $\beta_{ij} = \beta_{ij}(M)$ are called the \textit{graded Betti numbers} of $M$.
The number
$$\projdim(M) = \max\{i : \beta_{ij} \neq {0} ~ \text{for some} ~{j} \in \mathbb{Z}\}$$
is called the \textit{projective dimension} of $M$ and the number $$\reg(M) = \max\{j : \beta_{i,i+j} \neq {0}~~ \text{for some} ~ {i}\}$$ is called the \textit{regularity} of $M$. By Auslander-Buchsbam formula, we have
\[\depth M=\dim S-\projdim M.\]

\section{Main Results}\label{S3}

In this section we prove the main results of this paper.
\begin{Lemma}\label{2.1}
Let $m,n\geq 2$. Let $G$ be a graph on the vertex set $n$ and let $j$ be any vertex of $G$. Then
$$\ini_{<}(\mathfrak{J}_{G},x_{1j},\ldots, x_{mj})=(\ini_{<}(\mathfrak{J}_{G}),x_{1j},\ldots, x_{mj}).$$

\end{Lemma}
\begin{proof}
The proof is similar to \cite[Lemma 3.1]{CDI}. Clearly, we have
$$\ini_{<}(\mathfrak{J}_{G},x_{1j},\ldots, x_{mj})=\ini_{<}(\mathfrak{J}_{G\setminus \{j\}},x_{1j},\ldots, x_{mj})=(\ini_{<}(\mathfrak{J}_{G\setminus\{j\}}),x_{1j},\ldots, x_{mj}).$$
Therefore, we have to show that $(\ini_{<}(\mathfrak{J}_{G\setminus\{j\}}),x_{1j},\ldots, x_{mj})= (\ini_{<}(\mathfrak{J}_{G}),x_{1j},\ldots, x_{mj})$. Since $\mathfrak{J}_{G\setminus \{j\}}\subset\mathfrak{J}_{G}$, the inclusion $\subseteq$ is obvious. For the other inclusion, let $u$ be any generator of $\ini_{<}(\mathfrak{J}_{G})$. If there exist some $i$ such that $x_{ij}|u$, then obviously $u\in (\ini_{<}(\mathfrak{J}_{G\setminus\{j\}},x_{1j},\ldots, x_{mj}))$. Now suppose that $x_{ij}\nmid u$, for all $1\leq i\leq m $. This means that $u= u_{\pi}^{\kappa} x_{\kappa (l)k}x_{\kappa(k)l}$, for some admissible path $\pi$ from $k$ to $l$, which does not contains the vertex $j$. This implies that $\pi$ is a path in $G\setminus \{j\}$. Hence $u\in \ini_{<}(\mathfrak{J}_{G\setminus \{j\}})$. This completes the proof.
\end{proof}

\begin{Remark}\label{re}
Following the proof of Lemma~ \ref{2.1}, we can extend our result for any subset of vertices of $G$, that is, for any $A\subset[n],$ we have:
$$\ini_{<}(\mathfrak{J}_{G},\{x_{1j},\ldots, x_{mj}~|~j\in A\})=(\ini_{<}(\mathfrak{J}_{G}),\{x_{1j},\ldots, x_{mj}~|~j\in A\}).$$
\end{Remark}

First we observe that, in order to compute the depth and the regularity of $S/{\mathfrak{J}_{G}}$ or $S/\ini_{<}(\mathfrak{J}_{G})$ we may reduce to connected graphs. Indeed, if $G$ is disconnected and has the connected components $G_1,\ldots, G_c$, we have

\begin{equation}\label{E1}
 S/{\mathfrak{J}_{G}}\cong S_1/{\mathfrak{J}_{G_1}}\otimes_K \cdots \otimes_K S_c/{\mathfrak{J}_{G_c}}
\end{equation}

where $S_i= K[x_{1j},\ldots, x_{mj}: j\in V(G_i)].$

The above isomorphism is due to the fact that ${\mathfrak{J}_{G_1}},\ldots, {\mathfrak{J}_{G_c}} $ are generated in pairwise disjoint sets of variables. Equation (\ref{E1}) implies that

$$ \depth( S/{\mathfrak{J}_{G}})= \sum_{i=1}^{c} \depth(S_i/{\mathfrak{J}_{G_i}}), $$
and

$$ \reg( S/{\mathfrak{J}_{G}})= \sum_{i=1}^{c} \reg( S_i/{\mathfrak{J}_{G_i}}). $$
Same arguments work for the $\depth S/\ini_{<}(\mathfrak{J}_G)$ and $\reg S/\ini_{<}(\mathfrak{J}_G)$ if $G$ is disconnected.

Let us recall from Section \ref{S2}, that if $G$ is a generalized block graph, then $\mathcal{A}_i(G)$ is the collection of cut sets of $G$ of cardinality $i$, where $i=1,\ldots,\omega (G)-1$. We denote $a_i(G)=|\mathcal{A}_i(G)|$.

\begin{Theorem}\label{2.2}
Let $m$,$n\geq 2$ and let $G$ be a connected generalized block graph on the vertex set $[n]$. The following statements hold:
\begin{itemize}
\item[(a)] $\depth  S/\mathfrak{J}_G=\depth S/\ini_{<}(\mathfrak{J}_G) =n+(m-1)-\sum\limits_{\substack{i=2}}^{\omega(G)-1} (i-1)a_i(G);$
  \item[(b)] If $m\geq n$, then $\reg S/\mathfrak{J}_G = \reg S/\ini_{<}(\mathfrak{J}_G) =n-1$;
  \item[(c)] If $m< n$, then $\reg S/\mathfrak{J}_G \leq  \reg (S/\ini_{<}(\mathfrak{J}_G)) \leq n-1.$
\end{itemize}
\end{Theorem}

\begin{proof}
We split the proof of our theorem in two parts. In the first part we give the results for generalized binomial edge ideals, while in the second part we will present the results for their initial ideals.
This proof is based on the techniques used in the proof of \cite[Theorem 1.1]{EHH}, \cite[Theorem 3.2]{CDI} and \cite[Theorem 3.2]{KM}. Since $G$ is a chordal graph, then by Dirac's theorem \cite{Dirac}, $\Delta(G)$ is a quasi forest which means that there is a leaf order say $F_1,\ldots,F_r$, for the facets of $\Delta(G)$. Let $F_{t_1},\ldots,F_{t_q}$ be the branches of $F_r$. Since $G$ is a generalized block graph, the intersection of any pair of facets from $F_{t_1},\ldots,F_{t_q},F_r$ is the same set of vertices. Let $F_i\cap F_j=A$, for all $i,j\in\{t_1,\ldots,t_q,r\}$ and let $|A|=\alpha\geq 1$. Moreover, $F_r\cap F_k=\emptyset$ for all $k\neq t_1,\ldots,t_q,r$. This implies that $A\cap F_k=\emptyset$ for all $k\neq t_1,\ldots,t_q,r$. Hence $A$ is a $(q+1)$-minimal cut set of $G$. For any cut point set $\mathcal{T}$ of $G$, we have $A\nsubseteq \mathcal{T}$ if and only if $A\cap \mathcal{T}=\emptyset$ (see proof of Theorem 3.2 in \cite{KM}).

Let $\mathfrak{J}_G=J_1\cap J_2$, where $J_1=\bigcap\limits_{\substack{\mathcal{T}\subseteq [n] \\ A\cap \mathcal{T}=\emptyset}} P_{\mathcal{T}}(G)$ and $J_2=\bigcap\limits_{\substack{\mathcal{T}\subseteq [n] \\ A\subseteq \mathcal{T}}} P_{\mathcal{T}}(G)$. It follows that $J_{1}=J_{G'}$ where $G'$ is obtained from $G$ by replacing the cliques $F_{t_{1}},\ldots,F_{t_{q}}$ and $F_{r}$ by the clique on the vertex set $F_{r}\cup(\bigcup_{i=1}^q F_{t_{i}})$. Also, $J_{2}=(x_{1j},\ldots, x_{mj}~|~ j\in A)+J_{G''}$ where $G''$ is the restriction  of $G$ to the vertex set $[n]\setminus A$. We observe that $S/J_2=S_A/\mathfrak{J}_{G''}$, where $S_A=K[x_{1j},\ldots,x_{mj}~:~j\in [n]\setminus A]$ and $J_1+J_2=(x_{1j},\ldots, x_{mj}~:~ j\in A)+J_{G'_{[n]\setminus A}}$. Moreover, $G',G''$ and $G'_{[n]\setminus A}$ inherits the properties of $G$, that is, they are also generalized block graphs. According to the proof of Theorem 3.2 in \cite{KM}, we have $\mathcal{A}_i(G')=\mathcal{A}_i(G'_{[n]\setminus A})=\mathcal{A}_i(G)$ and $\mathcal{A}_i(G'')\subseteq\mathcal{A}_i(G)$, for all $i\neq \alpha$. But $\mathcal{A}_\alpha(G')=\mathcal{A}_\alpha(G'_{[n]\setminus A})=\mathcal{A}_\alpha(G)\setminus\{A\}$ and $\mathcal{A}_\alpha(G'')\subseteq\mathcal{A}_\alpha(G)\setminus\{A\}$.

This implies that
\begin{equation}\label{E2}
a_i(G')=a_i(G'_{[n]\setminus A})=a_i(G)
\end{equation}
and
\begin{equation}\label{E3}
a_i(G'')\leq a_i(G),~ \text{for all}~ i\neq \alpha.
\end{equation}
On the other hand,
\begin{equation}\label{E4}
a_\alpha(G')=a_\alpha(G'_{[n]\setminus A})=a_\alpha(G)-1
\end{equation}
and
\begin{equation}\label{E5}
a_\alpha(G'')\leq a_\alpha(G)-1.
\end{equation}

For $r>1$, we have the following exact sequence of $S$-modules
\begin{equation}\label{1}
0 \longrightarrow \frac{S}{\mathfrak{J}_{G}} \longrightarrow \frac{S}{J_{1}}\oplus\frac{S}{J_{2}}
\longrightarrow \frac{S}{J_{1}+ J_{2}}\longrightarrow 0.
\end{equation}
\\

$(a)$ We apply induction on the number $r$ of maximal cliques of $G$. If $r=1$, then $G$ is a simplex and the equality $\depth S/\mathfrak{J}_G =n+(m-1)$ follows by \cite[Theorem 4.4]{EHHQ}. For $r>1$, we can apply the inductive step. Since $G'$ has a smaller number of maximal cliques than $G$, it follows, by the inductive
hypothesis, that $$\depth(S/\mathfrak{J}_{G'})=n+(m-1)-\sum\limits_{\substack{i=2}}^{\omega(G')-1} (i-1)a_i(G')$$
$$\qquad\qquad\qquad\qquad\qquad\qquad\quad=n+(m-1)-\sum\limits_{\substack{i=2 \\ i\neq \alpha}}^{\omega(G')-1} (i-1)a_i(G')-(\alpha-1)a_{\alpha}(G')$$
$$\qquad\qquad\qquad\qquad\qquad\qquad\quad=n+(m-1)-\sum\limits_{\substack{i=2 \\ i\neq \alpha}}^{\omega(G)-1} (i-1)a_i(G)-(\alpha-1)(a_{\alpha}(G)-1).$$
In the last equation, we used equations (\ref{E2}) and (\ref{E4}). Therefore, we have

\begin{equation}\label{E6}
\depth(S/\mathfrak{J}_{G'})=n+(m-1)+\alpha-1-\sum\limits_{\substack{i=2}}^{\omega(G)-1} (i-1)a_i(G).
\end{equation}

Now, $G''$ has $q+1$ connected components, say, $H_1,\ldots,H_{q+1}$ on the vertex sets $[n_1],\ldots,[n_{q+1}]$. For $j=1,\ldots, {q+1}$, we set  $S^{j}_A= K[x_{1i},x_{2i} \ldots,x_{mi} : i \in [n_j]]$.  Also, $G''$ has less number of maximal cliques than $G$, so by the inductive hypothesis
$$\depth(S_A/\mathfrak{J}_{G''})=\sum_{j=1}^{q+1} (\depth(S_{A}^j/\mathfrak{J}_{H_j}))\quad\quad\qquad\qquad$$
$$\qquad\qquad\qquad\qquad=\sum_{j=1}^{q+1}(n_j+(m-1)-\sum\limits_{\substack{i=2}}^{\omega(H_j)-1} (i-1)a_i(H_j))$$
$$\qquad\qquad\qquad\qquad\qquad\qquad\quad=n-\alpha+(q+1)(m-1)-\sum\limits_{\substack{i=2 \\ i\neq \alpha}}^{\omega(G'')-1} (i-1)a_i(G'')-(\alpha-1)a_{\alpha}(G'')$$

$$\qquad\qquad\qquad\qquad\qquad\qquad\quad\geq n-\alpha+(q+1)(m-1)-\sum\limits_{\substack{i=2 \\ i\neq \alpha}}^{\omega(G)-1} (i-1)a_i(G)-(\alpha-1)(a_{\alpha}(G)-1).$$
In the last inequality, we used inequalities (\ref{E3}) and (\ref{E5}). Consequently, we get
\begin{equation}\label{E7}
\qquad\qquad\depth(S_A/\mathfrak{J}_{G''})\geq n+(q+1)(m-1)-1-\sum\limits_{\substack{i=2}}^{\omega(G)-1} (i-1)a_i(G).
\end{equation}

Moreover, $G'_{[n]\setminus A}$ is a generalized block graph with the number of maximal cliques less than $r$, hence by the inductive hypothesis, we have
$$\depth(S/\mathfrak{J}_{G'_{[n]\setminus A}})=n-\alpha+(m-1)-\sum\limits_{\substack{i=2}}^{\omega(G'_{[n]\setminus A})-1} (i-1)a_i(G'_{[n]\setminus A})$$
$$\qquad\qquad\qquad\qquad\qquad\qquad\quad=n-\alpha+(m-1)-\sum\limits_{\substack{i=2 \\ i\neq \alpha}}^{\omega(G'_{[n]\setminus A})-1} (i-1)a_i(G'_{[n]\setminus A})-(\alpha-1)a_{\alpha}(G'_{[n]\setminus A})$$
$$\qquad\qquad\qquad\qquad\qquad\qquad\quad=n-\alpha+(m-1)-\sum\limits_{\substack{i=2 \\ i\neq \alpha}}^{\omega(G)-1} (i-1)a_i(G)-(\alpha-1)(a_{\alpha}(G)-1).$$
In the last equation, we used equations (\ref{E2}) and (\ref{E4}). Therefore, we have
\begin{equation}\label{E8}
\depth(S/\mathfrak{J}_{G'_{[n]\setminus A}})=n+(m-1)-1-\sum\limits_{\substack{i=2}}^{\omega(G)-1} (i-1)a_i(G).
\end{equation}

By applying Depth lemma to our exact sequence $~(\ref{1})$, and taking into account equations (\ref{E6}), (\ref{E7}) and (\ref{E8}), we get
$$ \depth  S/\mathfrak{J}_G=n+(m-1)-\sum\limits_{\substack{i=2}}^{\omega(G)-1} (i-1)a_i(G).$$
\\

$(b)$  Let $m\geq n$. We apply again induction on $r$. If $r=1$, then $G$ is a simplex and the equality $\reg S/\mathfrak{J}_G =n-1$ is true; see \cite[Corollary 1]{CH} and \cite[Theorem 6.9]{CB}. Since $G'$ has smaller number of maximal cliques than $G$, it follows by the inductive
hypothesis, that $\reg S/\mathfrak{J}_{G'} =n-1$. We have seen in (a) that $G''$ has $q+1$ connected components, $H_1,\ldots,H_{q+1}$ on the vertex sets $[n_1],\ldots,[n_{q+1}]$. Since all $H_i$'s have less than $r$ maximal cliques, the inductive hypothesis implies that $\reg S/\mathfrak{J}_{H_i} =n_i-1$ for all $i=1,\ldots,q+1$. This implies that
$$\reg S/J_2=\reg S/\mathfrak{J}_{G''}=(n-\alpha)-(q+1).$$
It follows that $$\reg (S/J_1\oplus S/J_2)=\max\{\reg S/J_1, \reg S/J_2\}=n-1.$$
Next,
$$ \reg S/({J_1+ J_2})= \reg S/\mathfrak{J}_{G'_{[n]\setminus A}}=n-\alpha-1.$$
Here we applied the inductive hypothesis since $G'_{[n]\setminus A}$ has less number of maximal cliques than $G$. Consequently, it follows that

$$\reg(S/J_1\oplus S/J_2)> \reg S/({J_1+ J_2}),$$
which, by \cite[Proposition 18.6]{P}, yields $$\reg S/\mathfrak{J}_G=n-1.$$
\\

$(c)$ Since for any homogeneous ideal $I$ in a polynomial ring, we have $\beta_{ij}(I)\leq\beta_{ij}(\ini_<(I))$, it follows that $\reg(I)\leq\reg(\ini_<(I))$. Hence, to prove $(c)$, we just need to prove that $ \reg (S/\ini_{<}(\mathfrak{J}_G))\leq n-1$.

Now we continue in the same way as above to get the results for the initial ideal. We will use induction on number of maximal cliques of $G$.
We have $\ini_{<}(\mathfrak{J}_{G})=\ini_{<}(J_{1}\cap J_{2})$. By \cite[Lemma 1.3]{Concoa}, we get $\ini_{<}(J_{1}\cap J_{2})=\ini_{<}
(J_{1})\cap\ini_{<}(J_{2})$ if and only if $\ini_{<}(J_{1}+J_{2})=\ini_{<}(J_{1})+\ini_{<}(J_{2})$. But  $\ini_{<}(J_{1}+J_{2})=$$\ini_{<}(\mathfrak{J}_{G'}+(x_{1j},\ldots, x_{mj}~:~j\in A)+\mathfrak{J}_{G''})=$\\$\ini_{<}(\mathfrak{J}_{G'}+(x_{1j},\ldots, x_{mj}~:~j\in A))$. Hence, by Remark ~\ref{re}, we get $$\ini_{<}(J_{1}+J_{2})=\ini_{<}(\mathfrak{J}_{G'})+(x_{1j},\ldots, x_{mj}~:~j\in A)=\ini_{<}(J_{1})+\ini_{<}(J_{2}).$$ Therefore, we get $\ini_{<}(\mathfrak{J}_{G})=\ini_{<}(J_{1})\cap\ini_{<}(J_{2})$ and, consequently, we have an exact sequence of $S$-modules

\begin{equation}\label{2}
 0 \longrightarrow \frac{S}{\ini_{<}(\mathfrak{J}_{G})} \longrightarrow \frac{S}{\ini_{<}(J_{1})}\oplus\frac{S}{\ini_{<}(J_{2})}
\longrightarrow \frac{S}{\ini_{<}(J_{1}+ J_{2})}\longrightarrow 0,
\end{equation}
which is similar to exact sequence (\ref{1}).

By using again Remark~\ref{re}, we have $$\ini_{<}(\mathfrak{J}_{2})=\ini_{<}((x_{1j},\ldots, x_{mj}~:~j\in A),\mathfrak{J}_{G''})=(x_{1j},\ldots, x_{mj}~:~j\in A)+\ini_{<}(\mathfrak{J}_{G''}).$$
 Thus, we have actually the following exact sequence
\begin{equation}\label{3}
    0 \longrightarrow \frac{S}{\ini_{<}(\mathfrak{J}_{G})} \longrightarrow \frac{S}{\ini_{<}(\mathfrak{J}_{G'})}\oplus\frac{S}{(x_{1j},\ldots, x_{mj})+\ini_{
		<}(\mathfrak{J}_{G''})}\longrightarrow \frac{S}{(x_{1j},\ldots, x_{mj})+ \ini_{<}(\mathfrak{J}_{G'})}\longrightarrow 0.
\end{equation}
\\

$(a)$
If $r=1$, then $G$ is a simplex and the equality $\depth S/\mathfrak{J}_G= \depth S/\ini_<(\mathfrak{J}_G)$, follows since the ideal generated by all $2$-minors of the matrix $X$ is Cohen-Macaulay and its initial ideal shares the same property \cite{CB}. For $r>1$, since $G'$ has smaller number of maximal cliques than $G$, it follows by the inductive hypothesis, that
$$\depth(S/\mathfrak{J}_{G'})=\depth(S/\ini_<(\mathfrak{J}_{G'})).$$
By using equation (\ref{E6}), we have
$$\depth(S/\ini_<(\mathfrak{J}_{G'}))=n+(m-1)+\alpha-1-\sum\limits_{\substack{i=2}}^{\omega(G)-1} (i-1)a_i(G).$$
We have $S/((x_{1j},\ldots, x_{mj}~:~j\in A)+\ini_<(\mathfrak{J}_{G''}))\cong S_{A}/\ini_<(\mathfrak{J}_{G''})$. Since $G''$ is a graph on $n-\alpha$ vertices with $q+1$ connected components and satisfies our conditions, by induction
$$\depth S/((x_{1j},\ldots, x_{mj}~:~j\in A)+\mathfrak{J}_{G''})=\depth S/((x_{1j},\ldots, x_{mj}~:~j\in A)+\ini_<(\mathfrak{J}_{G''})).$$
By using inequality (\ref{E7}), we have
$$\depth S/((x_{1j},\ldots, x_{mj}~:~j\in A)+\ini_<(\mathfrak{J}_{G''})) \geq n+(q+1)(m-1)-1-\sum\limits_{\substack{i=2}}^{\omega(G)-1} (i-1)a_i(G).$$

We observe that $S/((x_{1j},\ldots, x_{mj}~:~j\in A)+\ini_<(\mathfrak{J}_{G'}))\cong S_{A}/\ini_<(\mathfrak{J}_{G'_{[n]\setminus A}})$. The inductive hypothesis implies that,
$$\depth(S/((x_{1j},\ldots, x_{mj}~:~j\in A)+\mathfrak{J}_{G'}))=\depth(S/((x_{1j},\ldots, x_{mj}~:~j\in A)+\ini_<(\mathfrak{J}_{G'}))).$$
By using equation (\ref{E8}), we have
$$\depth(S/((x_{1j},\ldots, x_{mj}~:~j\in A)+\ini_<(\mathfrak{J}_{G'})))=n+(m-1)-1-\sum\limits_{\substack{i=2}}^{\omega(G)-1} (i-1)a_i(G).$$
Hence, by applying Depth lemma to our exact sequence $~(\ref{3})$, we get
$$\depth S/\ini_<(\mathfrak{J}_{G})=n+(m-1)-\sum\limits_{\substack{i=2}}^{\omega(G)-1} (i-1)a_i(G).$$
\\

$(b)$ For $r=1$, $\mathfrak{J}_{G}$ and $\ini_<(\mathfrak{J}_{G})$ are Cohen-Macaulay ideals, see, for example, \cite{CB}. Since they share the same Hilbert series, then they have also the same regularity. By \cite[Corollary 1]{CH} and \cite[Theorem 6.9]{CB}, $\reg S/\mathfrak{J}_{G}=\min (m,n)-1=n-1$. Therefore $\reg (S/\ini_<(\mathfrak{J}_{G}))=n-1$. Since $G'$ has smaller number of maximal cliques than $G$, by the inductive
hypothesis, $\reg (S/\ini_<(\mathfrak{J}_G')) =n-1$. The graph $G''$ is the restriction of $G$ on the vertex set $[n]\setminus A$ and $G''$ has $q+1$ connected components $H_1,\ldots,H_{q+1}$ on the vertex sets $[n_1],\ldots,[n_{q+1}]$. Also, each $H_i$ has less number of maximal cliques than $G$, so by induction $\reg (S/\ini_<(\mathfrak{J}_{H_i})) =n_i-1$, for all $i=1,\ldots,q+1$. This implies that
$$\reg S/\ini_<(J_2)=\reg S/\ini_<(\mathfrak{J}_{G''})=(n-\alpha)-(q+1).$$
It follows that $$\reg (S/\ini_<(J_1)\oplus S/\ini_<(J_2))=\max\{\reg S/\ini_<(J_1), \reg S/\ini_<(J_2)\}=n-1.$$
Next,
$$ \reg S/\ini_<({J_1+ J_2})= \reg S/\ini_<(\mathfrak{J}_{G'_{[n]\setminus A}})=n-\alpha-1.$$
Hence we applied the inductive hypothesis since $G'_{[n]\setminus A}$ has less number of maximal cliques than $G$. Consequently, it follows that

$$\reg(S/\ini_<(J_1)\oplus S/\ini_<(J_2))> \reg S/\ini_<({\ini_<(J_1)+ \ini_<(J_2)}),$$
which by \cite[Proposition 18.6]{P}, yields $$\reg S/\ini_<(\mathfrak{J}_G)=n-1.$$


$(c)$ For $r=1$, by \cite[Corollary 1]{CH} and \cite[Theorem 6.9]{CB}, $\reg S/\mathfrak{J}_{G}=\min (m,n)-1=m-1$. Therefore $\reg (S/\ini_<(\mathfrak{J}_{G}))=m-1<n-1$. The graph $G'$ inherits the properties of $G$ and it has a smaller number of maximal cliques, so by the inductive hypothesis, $\reg (S/\ini_<(\mathfrak{J}_G')) \leq n-1$. Similarly, $\reg (S/\ini_<(\mathfrak{J}_{H_i})) \leq n_i-1$ for all $i=1,\ldots,q+1$, follow from the inductive hypothesis. This implies that
$$\reg S/\ini_<(J_2)=\reg S/\ini_<(\mathfrak{J}_{G''})\leq(n-\alpha)-(q+1).$$
It follows that
$$\reg (S/\ini_<(J_1)\oplus S/\ini_<(J_2))=\max\{\reg S/\ini_<(J_1), \reg S/\ini_<(J_2)\}\leq n-1.$$
Now,
$$ \reg S/\ini_<({J_1+ J_2})= \reg S/\ini_<(\mathfrak{J}_{G'_{[n]\setminus A}})\leq n-\alpha-1.$$
Here we applied the inductive hypothesis since $G'_{[n]\setminus A}$ has less number of maximal cliques than $G$. Consequently, it follows by \cite[Corollary 18.7]{P}, that
$$\reg S/\ini_<(\mathfrak{J}_G) \leq \max\{n-1,n-\alpha\}= n-1.$$
\end{proof}

When $G$ is a block graph, we obviously have $a_{i}(G)=0$, for all $i>1$, thus Theorem \ref{2.2} has the following consequences.
\begin{Corollary}\label{2.3}
Let $m$,$n\geq 2$ and let $G$ be a block graph on the vertex set $[n]$. The following statements holds:

\begin{itemize}
\item[(a)] $\depth  S/\mathfrak{J}_G=\depth S/\ini_{<}(\mathfrak{J}_G) =n+(m-1);$
  \item[(b)] If $m\geq n$, then $\reg S/\mathfrak{J}_G = \reg S/\ini_{<}(\mathfrak{J}_G) =n-1$;
  \item[(c)] If $m< n$, then $\reg S/\mathfrak{J}_G =\leq n-1$ and $ \reg S/\ini_{<}(\mathfrak{J}_G) \leq n-1.$
\end{itemize}
In particular, the above statements hold for a tree on the vertex set $[n]$.
\end{Corollary}
\begin{Corollary}\label{2.4}
Let $G$ be a block graph on $[n]$ and $m,n\geq 3$. Then the following are equivalent:
\begin{itemize}
  \item[(a)] $\mathfrak{J}_G$ is unmixed;
  \item[(b)] $\mathfrak{J}_G$ is Cohen-Macaulay;
  \item[(c)] $G=K_n$.
\end{itemize}

\end{Corollary}

\begin{proof}
$(a)\Rightarrow (c)$ By \cite[Proposition 4.1]{EHHQ}, $\mathfrak{J}_G$ is unmixed if and only if $$(c(\mathcal{T})-1)(m-1)=|\mathcal{T}|$$ for all subsets $\mathcal{T}$ with cut point property of $G$. If $G$ is not complete then we have at least one cut point set of cardinality $1$. Therefore, the above equality is possible if and only if  the empty set is the only cut point set of $G$. This is equivalent to saying that $G$ is complete.
\\
$(b)\Rightarrow (a)$ and $(c)\Rightarrow (b)$ are known.
\end{proof}

\begin{Corollary}\label{2.5}
 Let $m,n \geq 2.$ If $G$ is a path graph on the vertex set $[n]$, then $\reg S/\mathfrak{J}_G=\reg S/\ini_<(\mathfrak{J}_G)=n-1$.

\end{Corollary}

\begin{proof}
Let $G$ be a path graph and $J_G$ be its classical binomial edge ideal. We consider $J_G\subset S'=K[x_{ij}: i=1,2,~ 1\leq j\leq n]$. Then, by using \cite[Proposition 8]{Sara2}, we get $\reg S'/J_G \leq \reg S/\mathfrak{J}_G$. As $J_G$ is a complete intersection, we get $\reg S'/J_G=n-1$. This implies that $n-1\leq \reg S/\mathfrak{J}_G\leq \reg(S/\ini_<(\mathfrak{J}_G)\leq n-1$. Therefore, the statement follows.
\end{proof}

{}


\begin{thebibliography}{}

\bibitem{B} A. Banerjee, L. Nunez-Betancourt, {\em Graph connectivity and binomial edge ideals}, Proc. Amer. Math. Soc. {\bf 145}(2) (2017), 487-499
\bibitem{CB}  W. Bruns, A. Conca, {\em Gr$\ddot{o}$bner bases and determinantal ideals. Commutative algebra}, in singularities and computer algebra (Sinaia, 2002), 9--66, NATO Sci. Ser. II Math. Phys. Chem., 115, Kluwer Acad. Publ., Dordrecht, 2003.


\bibitem{CDI} F. Chaudhry, A. Dokuyucu, R. Irfan, {\em On the binomial edge ideals of block graphs},  An. $\c{S}$tiin$\c{t}$. Univ. ``Ovidius" Constan$\c{t}$a Ser. Math. {\bf 24}(2) (2016), 149--158.

\bibitem{Concoa} A. Conca, {\em Gorenstein ladder determinantal rings}, J. London Math. Soc.  {\bf 54}(3) (1996), 453--474.


\bibitem{CH} A. Conca, J. Herzog, {\em On the Hilbert function of determinantal rings and their canonical module}, Proc. Amer. Math. Soc.  {\bf 122}(3) (1994), 677--681.



\bibitem{Dirac} G. A. Dirac, {\em On rigid circuit graphs}, Abh. Math. Semin. Univ. Hamburg, {\bf 25} (1961), pp. 71--76.

\bibitem{AD} A. Dokuyucu, {\em Extremal Betti numbers of some classes of binomial edge ideals}, Math. Rep.(Bucur.) {\bf 17(67)}(4) (2015), 359--367.

%
\bibitem{EHH} V. Ene, J. Herzog, T. Hibi, {\em Cohen-Macaulay binomial edge ideals}, Nagoya Math. J. {\bf 204} (2011),  57--68.
%
\bibitem{EHHQ} V. Ene, J. Herzog, T. Hibi, A. A. Qureshi, {\em The binomial edge ideal of a pair of graphs}, Nagoya Math. J. {\bf 213} (2014),  105--125.

\bibitem{EZ} V. Ene, A. Zarojanu, {\em On the regularity of binomial edge ideals},  Math. Nachr. {\bf 288}(1) (2015), 19--24.
%
%
\bibitem{HHHKR} J. Herzog, T. Hibi, F. Hreinsdotir, T. Kahle, J. Rauh, {\em Binomial edge ideals and conditional independence statements},
 Adv. Appl. Math. \textbf{45} (2010), 317--333.
%


\bibitem{KM} D. Kiani, S. Saeedi Madani, {\em Some Cohen-Macaulay and Unmixed Binomial Edge Ideals}, Comm. Algebra, {\bf 43} (2015), 5434-5453.

\bibitem{KM1} D. Kiani, S. Saeedi Madani, {\em The Castelnuovo-Mumford regularity of binomial edge ideals}, J. Combin. Theory Ser. A, {\bf 139} (2016), 80--86.

\bibitem{KM2} D. Kiani, S. Saeedi Madani, {\em Binomial edge ideals with pure resolution}, Collect. Math. {\bf 65}(3) (2014), 331--340.

\bibitem{MM} K. Matsuda, S. Murai, {\em Regularity bounds for binomial edge ideals,} J. Commut. Algebra {\bf 5} (2013), 141--149.
%
%
\bibitem{Oh} M. Ohtani, {\em Graphs and Ideals generated by some $2$-minors},  Comm. Algebra {\bf 39} (2011), no. 3,  905--917


\bibitem{P} I. Peeva, \textit{Graded syzygies}, Algebra and Applications {\bf 14}, Springer, 2011

\bibitem{Rauh} J. Rauh, {\em Generalized binomial edge ideals}, Adv. Appl. Math. {\bf 50}(2013), 409--414.

\bibitem{Sara} S. Saeedi Madani, D. Kiani, {\em Binomial edge ideals of graphs}, Electron. J. Combin. {\bf 19} (2012), no. 2, \# P44.

\bibitem{Sara2} S. Saeedi Madani, D. Kiani, {\em On the binomial edge ideal of a pair of graphs}, Electron. J.Combin, {\bf 20} (2013), no. 1, \# P48.

\bibitem{Sara3} S. Saeedi Madani, D. Kiani, {\em Binomial edge ideal of regularity $3$}, arXiv:1706.09002.

%
\bibitem{SZ} S. Zafar and Z. Zahid, {\em On the Betti numbers of some classes of binomial edge ideals},  Electron. J.  Combin.  {\bf 20}(4)  (2013), \# P37.
%


\end{thebibliography}
\end{document}